\documentclass{amsart}
\usepackage{mathptmx}
\usepackage{amscd,amssymb,enumerate}
\usepackage[all]{xy}

\newtheorem{theorem}{Theorem}[section]
\newtheorem{lemma}[theorem]{Lemma}
\newtheorem{corollary}[theorem]{Corollary}
\newtheorem{proposition}[theorem]{Proposition}

\theoremstyle{definition}
\newtheorem{definition}[theorem]{Definition}
\newtheorem{example}[theorem]{Example}
\newtheorem{remark}[theorem]{Remark}

\numberwithin{equation}{theorem}

\def\fpt{\mathrm{fpt}}
\def\lct{\mathrm{lct}}
\def\ge{\geqslant}
\def\le{\leqslant}

\def\tilde{\widetilde}
\def\del{\partial}

\def\to{\longrightarrow}
\def\Proj{\operatorname{Proj}}
\renewcommand{\mod}{\,\operatorname{mod}\,}

\def\fraka{\mathfrak{a}}

\def\frakm{\mathfrak{m}}

\def\FF{\mathbf{F}}
\def\NN{\mathbf{N}}
\def\PP{\mathbf{P}}

\newcommand{\calK}{{\mathcal{K}}}

\newcommand{\calX}{{\mathcal{X}}}
\newcommand{\calO}{{\mathcal{O}}}

\newcommand{\conj}{\mathrm{conj}}
\newcommand{\Hyp}{\mathrm{Hyp}}
\newcommand{\Hom}{\mathrm{Hom}}
\newcommand{\Vect}{\mathrm{Vect}}
\newcommand{\Sym}{\mathrm{Sym}}
\newcommand{\Frob}{\mathrm{Frob}}

\newcommand{\ord}{\mathrm{ord}}
\newcommand{\R}{\mathrm{R}}

\newcommand{\ev}{\mathrm{ev}}
\newcommand{\coker}{\mathrm{coker}}
\newcommand{\can}{\mathrm{can}}
\newcommand{\Fil}{\mathrm{Fil}}

\newcommand{\dR}{\mathrm{dR}}

\begin{document}
\title{The $F$-pure threshold of a Calabi-Yau hypersurface}

\author{Bhargav Bhatt}
\address{School of Mathematics, Institute for Advanced Study, Einstein Drive, Princeton,
\newline NJ~08540, USA}
\email{bhargav.bhatt@gmail.com}

\author{Anurag K. Singh}
\address{Department of Mathematics, University of Utah, 155 South 1400 East, Salt Lake City,
\newline UT~84112, USA}
\email{singh@math.utah.edu}

\thanks{B.B.~was supported by NSF grants DMS~1160914 and DMS~1128155, and A.K.S.~by NSF grant DMS~1162585. Both authors were supported by NSF grant~0932078000 while in residence at MSRI.}

\date{\today}

\begin{abstract}
We compute the $F$-pure threshold of the affine cone over a Calabi-Yau hypersurface, and relate it to the order of vanishing of the Hasse invariant on the versal deformation space of the hypersurface.
\end{abstract}
\maketitle

\section{Introduction}

The $F$-pure threshold was introduced by Musta\c t\u a, Takagi, and Watanabe~\cite{TW, MTW}; it is a positive characteristic invariant, analogous to log canonical thresholds in characteristic zero. We calculate the possible values of the $F$-pure threshold of the affine cone over a Calabi-Yau hypersurface, and relate the threshold to the order of vanishing of the Hasse invariant, and to a numerical invariant introduced by van~der~Geer and Katsura in~\cite{GK}.

\begin{theorem}
\label{thm:intro}
Suppose $R=K[x_0,\dots,x_n]$ is a polynomial ring over a field $K$ of characteristic $p>n+1$, and $f$ is a homogeneous polynomial in $R$ defining a smooth Calabi-Yau hypersurface $X=\Proj R/fR$. Then the $F$-pure threshold of $f$ has the form
\[
\fpt(f)\ =\ 1-\frac{h}{p}\,,
\]
where $h$ is an integer with $0\le h\le\dim X$. If $p\ge n^2-n-1$, then the integer $h$ equals the order of vanishing of the Hasse invariant on the versal deformation space of $X \subset \PP^n$.
\end{theorem}

Hern\'andez has computed $F$-pure thresholds for binomial hypersurfaces~\cite{Hernandez1} and for diagonal hypersurfaces~\cite{Hernandez2}. The $F$-pure threshold is computed for a number of examples in~\cite[Section~4]{MTW}. Example~4.6 of that paper computes the $F$-pure threshold in the case of an ordinary elliptic curve, and raises the question for supersingular elliptic curves; this is answered by the above theorem.

The theory of $F$-pure thresholds is motivated by connections to log canonical thresholds; for simplicity, let $f$ be a homogeneous polynomial with rational coefficients. Using~$f_p$ for the corresponding prime characteristic model, one has
\[
\fpt(f_p)\le\lct(f)\qquad\text{ for all }p\gg0\,,
\]
where $\lct(f)$ denotes the log canonical threshold of $f$, and
\[
\lim_{p\to\infty}\fpt(f_p)=\lct(f)\,,
\]
see~\cite[Theorems~3.3,~3.4]{MTW}; this builds on the work of a number of authors, primarily Hara and Yoshida~\cite{Hara-Yoshida}. It is conjectured that $\fpt(f_p)$ and $\lct(f)$ are equal for infinitely many primes; see~\cite{MS} for more in this direction.

The $F$-pure threshold is known to be rational in a number of cases, including for principal ideals in an excellent regular local ring of prime characteristic~\cite{KLZ:JALG}. Other results on rationality include~\cite{BMS:MMJ, BMS:TAMS, BSTZ, Hara, ST}. For more on $F$-pure thresholds, we mention~\cite{BHMM, HMTW, MSV, Mustata, MY}.
 
\section{The $F$-pure threshold}

In~\cite{TW} the $F$-pure threshold is defined for a pair $(R,\fraka)$, where $\fraka$ is an ideal in an $F$-pure ring of prime characteristic. The following special case is adequate for this paper:

\begin{definition}
Let $(R,\frakm)$ be a regular local ring of characteristic $p>0$. For an element $f$ in $\frakm$, and integer $q=p^e$, we define
\[
\mu_f(q)\ :=\ \min\big\{k\in\NN\mid f^k\in\frakm^{[q]}\big\}\,,
\]
where $\frakm^{[q]}$ denotes the ideal generated by the elements $r^q$ for $r\in\frakm$. Note that $\mu_f(1)=1$, and that $1\le\mu_f(q)\le q$. Moreover, $f^{\mu_f(q)}\in\frakm^{[q]}$ implies that $f^{p\mu_f(q)}\in\frakm^{[pq]}$, and it follows that $\mu_f(pq)\le p\mu_f(q)$. Thus,
\[
\Big\{\frac{\mu_f(p^e)}{p^e}\Big\}_{e\ge0}
\]
is a non-increasing sequence of positive rational numbers; its limit is the \emph{$F$-pure threshold} of $f$, denoted $\fpt(f)$.
\end{definition}

By definition, $f^{\mu_f(q)-1}\notin\frakm^{[q]}$. Taking $p$-th powers, and using that $R$ is $F$-pure, we get
\[
f^{p\mu_f(q)-p}\ \notin\ \frakm^{[pq]}\,.
\]
Combining with the observation above, one has
\begin{equation}
\label{eqn:mu}
p\mu_f(q)-p+1\ \le\ \mu_f(pq)\ \le\ p\mu_f(q)\,.
\end{equation}
Note that this implies
\[
\mu_f(q)\ =\ \left\lceil{\frac{\mu_f(pq)}{p}}\right\rceil\qquad\text{ for each }q=p^e\,.
\]

The definition is readily adapted to the graded case where $R$ is a polynomial ring with homogeneous maximal ideal $\frakm$, and $f$ is a homogeneous polynomial.

\begin{remark}
\label{rmk:fptthickenings}
The numbers $\mu_f(p^e)$ may be interpreted in terms of thickenings of the hypersurface $f$ as follows. Let $K$ be a field of characteristic $p>0$, and $f$ a homogeneous polynomial of degree $d$ in $R=K[x_0,\dots,x_n]$. Fix integers $q=p^e$ and $t\le q$. The Frobenius iterate $F^e\colon R/fR\to R/fR$ lifts to a map $R/fR\to R/f^qR$; composing this with the canonical surjection $R/f^qR\to R/f^tR$, we obtain a map
\[
\tilde{F^e_t}\colon R/fR\to R/f^tR\,.
\]

Consider the commutative diagram with exact rows
\[
\CD
0@>>>R(-d)@>f>>R@>>>R/fR@>>>0\phantom{\,,}\\
@. @VVf^{q-t}F^eV @VVF^eV @VV\tilde{F^e_t}V \\
0@>>>R(-dt)@>f^t>>R@>>>R/f^tR@>>>0\,,
\endCD
\]
and the induced diagram of local cohomology modules
\begin{equation}
\label{equation:cd}
\CD
0@>>>H^n_\frakm(R/fR)@>>>H^{n+1}_\frakm(R)(-d)@>f>>H^{n+1}_\frakm(R)@>>>0\phantom{\,.}\\
@. @VV\tilde{F^e_t}V @VVf^{q-t}F^eV @VVF^eV \\
0@>>>H^n_\frakm(R/f^tR)@>>>H^{n+1}_\frakm(R)(-dt)@>f^t>>H^{n+1}_\frakm(R)@>>>0\,.
\endCD
\end{equation}
Since the vertical map on the right is injective, it follows that $\tilde{F^e_t}$ is injective if and only if the middle map is injective, i.e., if and only if the element
\[
f^{q-t}F^e\left(\left[\frac{1}{x_0\cdots x_n}\right]\right)\ =\ \left[\frac{f^{q-t}}{x_0^q\cdots x_n^q}\right]
\]
is nonzero, equivalently, $f^{q-t}\notin\frakm^{[q]}$. Hence $\tilde{F^e_t}\colon H^n_\frakm(R/fR)\to H^n_\frakm(R/f^tR)$ is injective if and only if $\mu_f(q)>q-t$.
\end{remark}

We note that the generating function of the sequence $\{\mu_f(p^e)\}_{e\ge1}$ is a rational function:

\begin{theorem}
Let $(R,\frakm)$ be a regular local ring of characteristic $p>0$, and let $f$ be an element of $\frakm$. Then the generating function
\[
G_f(z)\ :=\ \sum_{e\ge 0}\mu_f(p^e)z^e
\]
is a rational function of $z$ with a simple pole at $z=1/p$; the $F$-pure threshold of $f$ is
\[
\fpt(f)\ =\ \lim_{z\to 1/p}(1-pz)\,G_f(z)\,.
\]
\end{theorem}

\begin{proof}
Since the numbers $\mu_f(p^e)$ are unchanged when $R$ is replaced by its $\frakm$-adic completion, there is no loss of generality in assuming that $R$ is a complete regular local ring; the rationality of $\fpt(f)$ now follows from~\cite[Theorems~3.1,~4.1]{KLZ:JALG}. Let $\fpt(f)=a/b$ for integers $a$ and $b$. By~\cite[Proposition~1.9]{MTW}, one has
\[
\mu_f(p^e)\ =\ \lceil p^e\,\fpt(f)\rceil\ =\ \left\lceil{\frac{ap^e}{b}}\right\rceil
\qquad\text{for each }q=p^e\,.
\]
Suppose $ap^{e_0}\equiv ap^{e_0+e_1}\mod b$ for integers $e_0$ and $e_1$. Then $ap^{e_0}\equiv ap^{e_0+ke_1}\mod b$ for each integer~$k\ge0$. Hence there exists an integer $c$ such that
\begin{align*}
H(z)\ := \ \sum_{k\ge0}\mu_f(p^{e_0+ke_1})z^{e_0+ke_1}\ &=\ \sum_{k\ge0}\left\lceil{\frac{ap^{e_0+ke_1}}{b}}\right\rceil z^{e_0+ke_1}\\
&=\ \sum_{k\ge0}\frac{ap^{e_0+ke_1}+c}{b}\ z^{e_0+ke_1}\,,
\end{align*}
is a rational function of $z$ with a simple pole at $z=1/p$. Moreover,
\[
\lim_{z\to 1/p}(1-pz)\,H(z)\ =\ \frac{a}{be_1}\,.
\]
Partitioning the integers $e\ge e_0$ into the congruence classes module $e_1$, it follows that $G_f(z)$ is the sum of a polynomial in $z$ and $e_1$ rational functions of the form
\[
\sum_{k\ge 0}\mu_f(p^{\ell+ke})z^{\ell+ke}\,.
\]
The assertions regarding the pole and the limit now follow.
\end{proof}

The theorem holds as well in the graded setting.

\section{Preliminary results}

We record some elementary calculations that will be used later. Here, and in the following sections, $R$ will denote a polynomial ring $K[x_0,\dots,x_n]$ over a field $K$ of characteristic~$p>0$, and $\frakm$ will denote its homogeneous maximal ideal. By the \emph{Jacobian ideal} of a polynomial~$f$, we mean the ideal generated by the partial derivatives
\[
f_{x_i}\ :=\ \del\!f/\del x_i\quad\text{ for }\ 0\le i\le n\,.
\]
If $f$ is homogeneous of degree coprime to $p$, then the Euler identity ensures that $f$ is an element of the Jacobian ideal; this is then the defining ideal of the singular locus of $R/fR$.

\begin{lemma}
\label{lemma:jacobian}
Let $f$ be a homogeneous polynomial of degree $d$ in $K[x_0,\dots,x_n]$ such that the Jacobian ideal $J$ of $f$ is $\frakm$-primary. Then
\[
\frakm^{(n+1)(d-2)+1}\ \subseteq\ J\,.
\]
\end{lemma}

\begin{proof}
Since $J$ is $\frakm$-primary, it is a complete intersection ideal. As it is generated by forms of degree $d-1$, the Hilbert-Poincar\'e series of $R/J$ is
\[
P(R/J,t)\ =\ \frac{(1-t^{d-1})^{n+1}}{(1-t)^{n+1}}\ =\ (1+t+t^2+\dots+t^{d-2})^{n+1}\,.
\]
It follows that $R/J$ has no nonzero elements of degree greater than $(n+1)(d-2)$.
\end{proof}

\begin{lemma}
\label{lemma:colon}
Let $R=K[x_0,\dots,x_n]$ and $\frakm^{[q]}=(x_0^q,\dots,x_n^q)$. Then
\[
\frakm^{[q]}:_R\frakm^{(n+1)(d-2)+1}\ \subseteq\ \frakm^{[q]}+\frakm^{(n+1)(q-d+1)}\,,
\]
where $\frakm^i=R$ for $i\le0$.
\end{lemma}

\begin{proof}
We prove, more generally, that
\[
\frakm^{[q]}:_R\frakm^k\ =\
\begin{cases}
\frakm^{[q]}+\frakm^{nq+q-n-k}&\text{ if }\ 0\le k\le nq+q-n\,,\\
R&\text{ if }\ k\ge nq+q-n\,.
\end{cases}
\]
Suppose $r$ is a homogeneous element of $\frakm^{[q]}:_R\frakm^k$. Computing the local cohomology module~$H^{n+1}_\frakm(R)$ via a \v Cech complex on $x_0,\dots,x_n$, the element
\[
\left[\frac{r}{x_0^q\cdots x_n^q}\right]\ \in\ H^{n+1}_\frakm(R)
\]
is annihilated by $\frakm^k$, and hence lies in ${\left[H^{n+1}_\frakm(R)\right]}_{\ge -n-k}$. If $r\notin\frakm^{[q]}$, then
\[
\deg r-(n+1)q\ \ge\ -n-k\,,
\]
i.e., $r\in\frakm^{nq+q-n-k}$. The pigeonhole principle implies that $\frakm^{nq+q-n}$ is contained in $\frakm^{[q]}$, which gives the rest.
\end{proof}

\begin{lemma}
\label{lemma:bound}
Let $f$ be a homogeneous polynomial of degree $d$ in $K[x_0,\dots,x_n]$, such that the Jacobian ideal of $f$ is $\frakm$-primary. If $\mu_f(q)$ is not a multiple of $p$, then 
\[
\mu_f(q)\ \ge\ \frac{(n+1)(q+1)-nd}{d}\,.
\]
\end{lemma}

\begin{proof}
Set $k:=\mu_f(q)$, i.e., $k$ is the least integer such that
\[
f^k\ \in\ \frakm^{[q]}\,.
\]
Applying the differential operators $\del/\del x_i$ to the above, we see that
\[
kf^{k-1}f_{x_i}\ \in\ \frakm^{[q]}\qquad\text{ for each }i\,,
\]
since $\del/\del x_i$ maps elements of $\frakm^{[q]}$ to elements of $\frakm^{[q]}$. As $k$ is nonzero in $K$, one has
\[
f^{k-1}J\ \subseteq\ \frakm^{[q]}\,,
\]
where $J$ is the Jacobian ideal of $f$. Lemma~\ref{lemma:jacobian} now implies that
\[
f^{k-1}\frakm^{(n+1)(d-2)+1}\ \subseteq\ \frakm^{[q]}\,.
\]
By Lemma~\ref{lemma:colon}, we then have
\[
f^{k-1}\ \in\ \frakm^{[q]}+\frakm^{(n+1)(q-d+1)}\,.
\]
But $f^{k-1}\notin\frakm^{[q]}$ by the minimality of $k$, so $f^{k-1}$ has degree at least $(n+1)(q-d+1)$, i.e.,
\[
d(k-1)\ \ge\ (n+1)(q-d+1)\,;
\]
rearranging the terms, one obtains the desired inequality
\[
k\ \ge\ \frac{(n+1)(q+1)-nd}{d}\,.
\qedhere
\]
\end{proof}

\begin{lemma}
\label{lemma:patterns}
Let $f$ be a homogeneous polynomial of degree $d$ in $K[x_0,\dots,x_n]$, such that the Jacobian ideal of $f$ is $\frakm$-primary. 
\begin{enumerate}[\quad\rm(1)]
\item If \ $\displaystyle{\frac{\mu_f(q)-1}{q-1}=\frac{n+1}{d}}$ \ for some $q=p^e$, then \ $\displaystyle{\frac{\mu_f(pq)-1}{pq-1}=\frac{n+1}{d}}$.

\item Suppose $p\ge nd-d-n$. If \ $\displaystyle{\frac{\mu_f(q)}{q}<\frac{n+1}{d}}$ \ for some $q=p^e$, then $\mu_f(pq)=p\mu_f(q)$.
\end{enumerate}
\end{lemma}

\begin{proof}
(1) Since $f^{\mu_f(q)-1}$ has degree $(q-1)(n+1)$ and is not an element of $\frakm^{[q]}$, it must generate the socle in $R/\frakm^{[q]}$. But then
\[
\left(f^{\mu_f(q)-1}\right)^{\frac{pq-1}{q-1}}
\]
generates the socle in $R/\frakm^{[pq]}$, so
\[
\mu_f(pq)-1\ =\ \big(\mu_f(q)-1\big)\left(\frac{pq-1}{q-1}\right)\,.
\]

For (2), suppose that $\mu_f(pq)<p\mu_f(q)$. Then $\mu_f(pq)$ is not a multiple of $p$ by~\eqref{eqn:mu}. Lemma~\ref{lemma:bound} thus implies that
\[
(n+1)(pq+1)-nd\ \le\ d\mu_f(pq)\,.
\]
Combining with $\mu_f(pq)\le p\mu_f(q)-1$ and $d\mu_f(q)\le q(n+1)-1$, we obtain
\[
p\ \le\ nd-d-n-1\,,
\]
which contradicts the assumption on $p$.
\end{proof}

We next prove a result on the injectivity of the Frobenius action on negatively graded components of local cohomology modules:

\begin{theorem}
\label{thm:highdeg}
Let $K$ be a field of characteristic $p>0$. Let $f$ be a homogeneous polynomial of degree $d$ in $R=K[x_0,\dots,x_n]$, such that the Jacobian ideal of $f$ is primary to the homogeneous maximal ideal $\frakm$ of $R$. 

If $p\ge nd-d-n$, then the Frobenius action below is injective:
\[
F\colon{\left[H^n_\frakm(R/fR)\right]}_{<0}\to{\left[H^n_\frakm(R/fR)\right]}_{<0}\,.
\]
\end{theorem}

\begin{proof}
Using~\eqref{equation:cd} in the case $t=1$ and $e=1$, and restricting to the relevant graded components, we have the diagram with exact rows
\[
\CD
0@>>>{\left[H^n_\frakm(R/fR)\right]}_{\le-1}@>>>{\left[H^{n+1}_\frakm(R)\right]}_{\le-d-1}@>>>\cdots\phantom{\,.}\\
@. @VVFV @VVf^{p-1}FV\\
0@>>>{\left[H^n_\frakm(R/fR)\right]}_{\le-p}@>>>{\left[H^{n+1}_\frakm(R)\right]}_{\le-d-p}@>>>\cdots\,.
\endCD
\]
Thus, it suffices to prove the injectivity of the map
\[
f^{p-1}F\colon{\left[H^{n+1}_\frakm(R)\right]}_{\le-d-1}\to{\left[H^{n+1}_\frakm(R)\right]}_{\le-d-p}\,.
\]
A homogeneous element of ${[H^{n+1}_\frakm(R)]}_{\le-d-1}$ may be written as
\[
\left[\frac{g}{(x_0\cdots x_n)^{q/p}}\right]
\]
for some $q$, where $g\in R$ is homogeneous of degree at most $(n+1)q/p-d-1$. Suppose
\[
f^{p-1}F\left( \left[\frac{g}{(x_0\cdots x_n)^{q/p}}\right] \right)\ =\ 0\,,
\]
then it follows that $f^{p-1}g^p\in\frakm^{[q]}$. Let $k$ be the least integer with
\[
f^kg^p\ \in\ \frakm^{[q]}\,,
\]
and note that $0\le k\le p-1$. If $k$ is nonzero, then applying $\del/\del x_i$ we see that
\[
f^{k-1}g^pJ\ \subseteq\ \frakm^{[q]}\,.
\]
Lemma~\ref{lemma:jacobian} and Lemma~\ref{lemma:colon} show that
\[
f^{k-1}g^p\ \in\ \frakm^{[q]}+\frakm^{(n+1)(q-d+1)}\,.
\]
Since $f^{k-1}g^p\notin\frakm^{[q]}$, we must have
\[
\deg f^{k-1}g^p\ \ge\ (n+1)(q-d+1)\,.
\]
Using $k\le p-1$ and $\deg g^p\le q(n+1)-pd-p$, this gives
\[
nd-d-n-1\ \ge\ p\,,
\]
contradicting the assumption on $p$. It follows that $k=0$, i.e., that $g^p\in\frakm^{[q]}$. But then
\[
\left[\frac{g}{(x_0\cdots x_n)^{q/p}}\right]\ =\ 0
\]
in $H^{n+1}_\frakm(R)$, which proves the desired injectivity.
\end{proof}

\begin{remark}
Theorem~\ref{thm:highdeg} is equivalent to the following geometric statement: if $X$ is a smooth hypersurface of degree $d$ in $\PP^n$, then the map
\[
H^{n-1}(X,\calO_X(j))\to H^{n-1}(X,\calO_X(jp))\,,
\]
induced by Frobenius map on $X$, is injective for $j<0$ and $p\ge nd-d-n$. This statement indeed admits a geometric proof based on the Deligne-Illusie method \cite{DI}. One views the de Rham complex $\Omega^*_{X/K}$ as an $\calO_{X^{(1)}}$-complex, where $X^{(1)}$ is the Frobenius twist of $X$ over~$K$, and twists it over the latter with $\calO_{X^{(1)}}(j)$. For $p>n-1$, the Deligne-Illusie decomposition $\Omega^*_{X/K} \simeq \oplus_i \Omega^i_{X^{(1)}/K}[-i]$, which is available as $X$ clearly lifts to $W_2(K)$, reduces the above injectivity statement to proving $H^{n-1-i}(X,\Omega^i_{X/K}(jp)) = 0$ for $i > 0$ and $j < 0$. If $p \ge nd-d-n$, this vanishing can be proven using standard sequences (details omitted).
\end{remark}

\section{Calabi-Yau hypersurfaces}

We get to the main theorem; see below for the definition of the Hasse invariant.

\begin{theorem}
\label{thm:CalabiYau}
Let $K$ be a field of characteristic $p>0$, and $n$ a positive integer. Let $f$ be a homogeneous polynomial of degree $n+1$ in $R=K[x_0,\dots,x_n]$, such that the Jacobian ideal of~$f$, i.e., the ideal $(f_{x_0},\dots,f_{x_n})$, is primary to the homogeneous maximal ideal of $R$. Then:
\begin{enumerate}[\quad\rm(1)]
\item $\mu_f(p)=p-h$, where $h$ is an integer with $0\le h\le n-1$,
\item $\mu_f(pq)=p\mu_f(q)$ for all $q=p^e$ with $q\ge n-1$.
\item If $p\ge n-1$, then $\displaystyle{G_f(z)=\frac{1-hz}{1-pz}}$ \ and \ $\displaystyle{\fpt(f)=1-\frac{h}{p}}$\,, where $0\le h\le n-1$.
\item Set $X=\Proj R/fR$. If $p\ge n^2-n-1$, then the integer $h$ in $(1)$ is the order of vanishing of the Hasse invariant on the versal deformation space of $X \subset \PP^n$.
\end{enumerate}
\end{theorem}

The deformation space in (4) refers to embedded deformations of $X\subset\PP^n$; if $n\ge 5$, this coincides with the versal deformation space of $X$ as an abstract variety (see Remark~\ref{rmk:VersalDefHyp}). The following example from~\cite{Hernandez2} shows that all possible values of $h$ from (1) above are indeed attained:

\begin{example}
\label{ex:CalabiYau}
Consider $f=x_0^{n+1}+\dots+x_{n+1}^{n+1}$ over a field of prime characteristic $p$ not dividing $n+1$. Let $h$ be an integer such that $p\equiv h+1\mod n+1$ and $0\le h\le n-1$. Then
\[
\fpt(f)\ =\ 1-h/p\,,
\]
for a proof, see~\cite[Theorem~3.1]{Hernandez2}.
\end{example}

\begin{proof}[Proof of Theorem~\ref{thm:CalabiYau}]
If $\mu_f(p)=p$, then Lemma~\ref{lemma:patterns}\,(1) shows that $\mu_f(q)=q$ for all $q$, and assertions (1--3) follow. Assume that $\mu_f(p)<p$. Lemma~\ref{lemma:bound} gives $\mu_f(p)\ge p-n+1$, which proves (1). As $\mu_f(p)\le p-1$, it follows~that
\[
\mu_f(q)\ \le\ q-q/p\qquad\text{ for each }q=p^e\,.
\]
If $\mu_f(pq)<p\mu_f(q)$, then $\mu_f(pq)$ is not a multiple of $p$ by~\eqref{eqn:mu}. Lemma~\ref{lemma:bound} now implies
\[
\mu_f(pq)\ \ge\ pq-n+1\,,
\]
and combining with $\mu_f(pq)\le p\mu_f(q)-1\le pq-q-1$, we see that $pq-n+1\le pq-q-1$, i.e., that $q\le n-2$. This completes the proof of (2), and then (3) follows immediately.

The proof of (4) and the surrounding material occupy the rest of this section.
\end{proof}

\subsection*{The Hasse invariant}

We briefly review the construction of the Hasse invariant for suitable families of varieties in characteristic $p$. Fix a proper flat morphism $\pi\colon\calX\to S$ of relative dimension $N$ between noetherian $\FF_p$-schemes. Assume that the formation of~$\R^i\pi_*\calO_\calX$ is compatible with base change, and that $\omega := \omega_{\calX/S} := \R^N\pi_*\calO_\calX$ is a line bundle; the key example is a family of degree $(n+1)$ hypersurfaces in $\PP^n$. The standard diagram of Frobenius twists of $\pi$ takes the shape
\[
\xymatrix{
\calX \ar@/_/[ddr]_{\pi} \ar@/^/[drr]^{\Frob_\calX} \ar[dr]^-{\Frob_\pi} & & \\
& \calX^{(1)} \ar[d]^-{\pi^{(1)}} \ar[r]^-{\Frob_S} & \calX \ar[d]^-{\pi} \\
& S \ar[r]^-{\Frob_S} & S\,,}
\]
where the square is Cartesian. Our assumption on $\pi$ shows that
\[
\omega_{\calX^{(1)}/S}\ :=\ \R^N \pi_*\calO_{\calX^{(1)}}\ \simeq\ \Frob_S^*\omega\ \simeq\ \omega^p\,.
\]
Using this isomorphism, we define:

\begin{definition}
The \emph{Hasse invariant} $H$ of the family $\pi$ is the element in 
\[
\Hom(\omega_{\calX^{(1)}/S},\omega)\ \simeq\ \Hom(\Frob_S^*\omega,\omega)\ \simeq\ \Hom(\omega^p,\omega)\ \simeq\ H^0(S,\omega^{1-p})\,,
\]
defined by pullback along the relative Frobenius map $\Frob_\pi\colon\calX\to\calX^{(1)}$.
\end{definition}

\begin{remark} 
\label{rmk:hasseflatbasechange}
The formation of the relative Frobenius map $\Frob_\pi\colon\calX\to\calX^{(1)}$ is compatible with base change on $S$. It follows by our assumption on $\pi$ that the formation of $H$ is also compatible with base change. In particular, given a flat morphism $g\colon S' \to S$ and a point $s'\in S'$, the order of vanishing of $H$ at $s'$ coincides with that at $g(s')$. Thus, in proving Theorem~\ref{thm:CalabiYau}\,(4), we may assume that $K$ is perfect.
\end{remark}

To analyze $H$, fix a point $s$ in $S$ and an integer $t\ge0$. Write $t[s]\subset S$ for the order $t$ neighbourhood of $s$, and let $t\!\calX_s\subset\calX$ and $t\!\calX^{(1)}_s\subset\calX^{(1)}$ be the corresponding neighbourhoods of the fibres of $\pi$ and $\pi^{(1)}$. The map $\Frob_\pi$ induces maps $t\!\calX_s\to t\!\calX^{(1)}_s$, and hence maps
\[
\phi_t\colon H^N\big(t\!\calX_{s}^{(1)},\ \calO_{t\!\calX_s^{(1)}}\big)\to H^N\big(t\!\calX_s,\ \calO_{t\!\calX_s}\big)\,.
\]
The order of vanishing of $H$ at $s$ is, by definition, the maximal $t$ such that this map is zero. In favourable situations, one can give a slightly better description of this integer:

\begin{lemma}
\label{lem:identifyhasse}
If the map
\[
\psi_t\colon H^N\big(\calX_s,\ \calO_{X_s}\big) \to H^N\big(t\!\calX_s,\ \calO_{t\!\calX_s}\big)
\]
induced by $\Frob_{\calX}$ is nonzero for some $t\le p$, then the minimal such $t$ is $\ord_s H+1$.
\end{lemma}

\begin{proof}
For $t \le p$, by the base change assumption on $\R^N\pi_*\calO_{\calX}$, one has 
\[
H^N\big(\calX_s,\ \calO_{X_s}\big)\otimes_{\kappa(s)}\calO_{t[s]}\ \simeq\ H^N\big(t\!\calX^{(1)}_s,\ \calO_{t\!\calX^{(1)}_s}\big)\,,
\]
where $\calO_{t[s]}$ is viewed as $\kappa(s)$-algebra via the composite
\[
\CD
\kappa(s)@>{\Frob_S}>>\calO_{\Frob_S^{-1}[s]}@>{\can}>>\calO_{t[s]}\,,
\endCD
\]
and the isomorphism is induced by the base change $\calX^{(1)} \to \calX$ of $\Frob_S\colon S \to S$. Hence, for such $t$, by adjunction, the map $\phi_t$ induced by $\Frob_\pi$ is nonzero if and only the map $\psi_t$ induced by $\Frob_\calX$ is nonzero. But $\ord_s H$ is the maximal integer $t$ with $\phi_t = 0$.
\end{proof}

It is typically hard to calculate $H$, or even bound its order of vanishing. However, for families of Calabi-Yau hypersurfaces, we have the following remarkable theorem due to Deuring and Igusa; see~\cite{Igusa} for $n=2$, and Ogus~\cite[Corollary~3]{Ogus} in general:

\begin{theorem}
\label{thm:ogusigusadeuring}	
Let $\pi\colon\calX\to\Hyp_{n+1}$ be the universal family of Calabi-Yau hypersurfaces in $\PP^n$. For any point $[Y]\in\Hyp_{n+1}(K)$ corresponding to a smooth hypersurface $Y\subset\PP^n$, we have $\ord_{[Y]}(H)\le n-1$ if $n\le p$.
\end{theorem}

Ogus's proof relies on crystalline techniques: he relates $\ord_{[Y]}(H)$ to the relative position of the conjugate and Hodge filtrations on a crystalline cohomology group of $Y$ (following an idea of Katz), and then exploits the natural relation between the Hodge filtration and deformation theory of $Y$. His result will not be used in proving Theorem~\ref{thm:CalabiYau}; in fact, our methods yield an alternative proof of Theorem~\ref{thm:ogusigusadeuring} avoiding crystalline methods under a mild additional constraint on the prime $p$, see Remark~\ref{remark:Ogus}.

\begin{remark}
\label{rmk:VersalDefHyp}
The universal family $\pi\colon\calX\to\Hyp_{n+1}$ is, in fact, versal at $[X]$ if $n \ge 5$ so $\ord_{[X]}(H)$, i.e., the order of vanishing of $H$ at $[X]\in\Hyp_{n+1}(K)$, is completely intrinsic to $X$. To see versality, it suffices to show that the map $\Hom(I_X/I_X^2, \calO_X)\to H^1(X,T_X)$ obtained from the adjunction sequence 
\[
\CD
0@>>>I_X/I_X^2@>>>\Omega^1_{\PP^n}|_X@>>>\Omega^1_X@>>>0
\endCD
\]
is surjective, and that $H^2(X,T_X)=0$. By the long exact sequence, it suffices to show the vanishing of $H^1(X,T_{\PP^n}|_X)$ and $H^2(X,T_X)$. The Euler sequence
\[ 
\CD
0@>>>\calO_{\PP^n}@>>>\calO_{\PP^n}(1)^{\oplus n+1}@>>>T_{\PP^n}@>>>0
\endCD
\]
restricted to $X$ immediately shows that $H^i(X,T_{\PP^n}|_X) = 0$ for $i = 1,2$ if $n \ge 5$; here we use that $H^i(X,\calO_X(j)) = 0$ for $0 < i < n-1$ and all $j$. The cohomology sequence for the dual of the adjunction sequence then shows that $H^2(X,T_X)=0$.
\end{remark}

\subsection*{The universal family}

Fix notation as in Theorem~\ref{thm:CalabiYau}, with $K$ a perfect field. Let
\[
\Hyp_{n+1}\ := \ \PP\big(H^0(\PP^n,\calO_{\PP^n}(n+1))^\vee\big)
\]
be the space of hypersurfaces of degree $(n+1)$ in $\PP^n$; we follow Grothendieck's conventions regarding projective bundles. Let $\pi\colon\calX\to\Hyp_{n+1}$ be the universal family, and let $\ev\colon\calX\to\PP^n$ be the evaluation map. Informally, $\calX$ parametrizes pairs $(x,Y)$ where $x\in\PP^n$ and $Y \in \Hyp_{n+1}$ is a degree $(n+1)$ hypersurface containing $x$. This description shows that $\ev\colon\calX\to\PP^n$ is a projective bundle, and we can formally write it as $\PP(\calK^\vee)\to\PP^n$, where $\calK\in\Vect(\PP^n)$ is defined as 
\[
\calK\ :=\ \ker\big(H^0(\PP^n,\calO_{\PP^n}(n+1))\otimes\calO_{\PP^n}\ \to\ \calO_{\PP^n}(n+1)\big)\,,
\]
and the map is the evident one. The resulting map
\[
\calX\to\PP\big(H^0(\PP^n,\calO_{\PP^n}(n+1))^\vee\big)
\]
is identified with $\pi$. Our chosen hypersurface $X \in \PP^n$ gives a point $[X] \in \Hyp_{n+1}(K)$ with $X' := \pi^{-1}([X])$ mapping isomorphically to $X$ via $\ev$. For an integer $t \ge 1$, let $tX \subset\PP^n$ and $tX' \subset \calX$ be the order $t$ neighbourhoods of $X \subset\PP^n$ and $X' \subset \calX$ respectively. 

\begin{lemma}
\label{lem:pullbackevinjective}
The map $H^{n-1}(tX,\calO_{tX}) \to H^{n-1}(tX',\calO_{tX'})$ is injective for all $t$.
\end{lemma}

\begin{proof}
Let $V = H^0(\PP^n,\calO_{\PP^n}(1))$, so $\PP^n = \PP(V)$ and $\Hyp_{n+1} = \PP(\Sym^{n+1}(V)^\vee)$. For each~$t$, the sheaf $\calO_{tX}$ admits a filtration defined by powers of the ideal defining $X \subset tX$, and similarly for $tX'$. The map $tX' \to tX$ is compatible with this filtration as it sends $X' \subset tX'$ to~$X \subset tX$. Hence, it suffices to check that the induced map
\[
\phi_j\colon H^{n-1}(X,I_X^j/I_X^{j+1})\to H^{n-1}(X',I_{X'}^j/I_{X'}^{j+1})
\]
is injective for each $j \ge 0$. Fix an isomorphism $\det(V)\simeq K$ and $f \in R_{n+1}$ defining $X$. These choices determine isomorphisms $I_X \simeq \calO_{\PP^n}(-n-1) \simeq K_{\PP^n}$, and hence an isomorphism
\[
H^{n-1}(X,O_X)\ \simeq\ H^n(\PP^n,I_X)\ \simeq\ H^n(\PP^n,K_{\PP^n})\ \simeq\ K\,.
\]
Tensoring the exact sequence $0\to I_X\to\calO_{\PP^n}\to\calO_X\to0$ with $I_X^j$ and using Serre duality shows that
\begin{align*}
H^{n-1}\big(X,I_X^j/I_X^{j+1}\big)\ &=\ \ker\Big(H^n\big(\PP^n,I_X^{j+1}\big)\to H^n\big(\PP^n,I_X^j\big)\Big)\\
&=\ \coker\Big(H^0\big(\PP^n,\calO_{\PP^n}((j-1)(n+1))\big)\to H^0\big(\PP^n,\calO_{\PP^n}(j(n+1))\big)\Big)^\vee\\ 
&=\ \coker\Big(\Sym^{(j-1)(n+1)}(V) \stackrel{f}{\to} \Sym^{j(n+1)}(V)\Big)^\vee. 
\end{align*}

As $X' \subset \calX$ is a fibre of $\pi$, one has $tX' = \pi^{-1}(t[X])$, where $t[X] \subset \Hyp_{n+1}$ is the order $t$ neighbourhood of $[X] \in \Hyp_{n+1}(K)$. Using flatness of $\pi$ and the aforementioned isomorphism $H^{n-1}(X,O_X) \simeq K$, we get
\[
H^{n-1}(X',\ I_{X'}^j/I_{X'}^{j+1})\ =\ \big(\Sym^j(\Sym^{n+1}(V)/(f))\big)^\vee.
\]
One can check that the pullback $\phi_j$ above is dual to the map induced by the composition map $\Sym^j(\Sym^{n+1}(V)) \to \Sym^{j(n+1)}(V)$ by passage to the appropriate quotients. In particular, the dual map is surjective, so $\phi_j$ is injective.
\end{proof}

\begin{proof}[Proof of Theorem~\ref{thm:CalabiYau}\,(4)]
By Remark~\ref{rmk:hasseflatbasechange}, we may assume that the field $K$ is perfect. Fix some $1\le t\le p$. We get a commutative diagram
\[
\CD
X' @>i>> tX' @>j>> pX' @>a>> \Frob_\calX^*X' @>c>> X'\\
@VVV @VVV @VVV @VVV @VVV\\
X @>k>> tX @>\ell>> pX @>b>> \Frob_{\PP^n}^*X @>d>> X
\endCD
\]
Here all vertical maps are induced by $\ev\colon\calX\to\PP^n$, the maps $c$ and $d$ are induced by the Frobenius maps on $\calX$ and $\PP^n$ respectively, and $i$, $j$, $k$, $\ell$, $a$ and $b$ are the evident closed immersions; the map $b$ is an isomorphism as $X \subset \PP^n$ is a Cartier divisor. In particular, the composite map $X' \to X'$ and $X \to X$ obtained from each row are the Frobenius maps on~$X'$ and $X$ respectively. Passing to cohomology gives a commutative diagram
\[
\CD
H^{n-1}(X,\calO_X) @= H^{n-1}(X',\calO_{X'})\\
@VV{a_t}V @VV{b_t}V \\
H^{n-1}(tX,\calO_{tX}) @>{c_t}>> H^{n-1}(tX',\calO_{tX'})
\endCD
\]
where $a_t$ and $b_t$ are induced by the Frobenius maps on $\PP^n$ and $\calX$ respectively, while $c_t$ is injective by Lemma~\ref{lem:pullbackevinjective}. To finish the proof, observe that Lemma~\ref{lem:identifyhasse} shows that $h+1$ is the minimal value of $t$ for which $b_t$ is injective. Since $c_t$ is injective as well, this is also the minimal value of $t$ for which $a_t$ is injective. It follows by Remark~\ref{rmk:fptthickenings} and Theorem~\ref{thm:highdeg} that $\mu_f(p) = p - (h+1) + 1 = p-h$. 
\end{proof}

\begin{remark}
\label{remark:Ogus}
We recover Ogus's result, Theorem~\ref{thm:ogusigusadeuring}, for primes $p\ge n^2-n-1$; this is immediate from Theorem~\ref{thm:CalabiYau}\,(1) and (4).
\end{remark}

We conclude by giving a de~Rham interpretation for $\fpt(f)$; write $\Fil^\conj_\bullet$ and $\Fil^\bullet_H$ for the increasing conjugate and decreasing Hodge filtrations on $H^{n-1}_\dR(X)$ respectively. Then:

\begin{corollary}
One has $\fpt(f)=1-a/p$, where $a$ is the largest $i$ such that 
\[
\Frob_X^* H^{n-1}_\dR(X)\ \simeq\ \Fil^\conj_0(H^{n-1}_\dR(X))\ \subset\ \Fil^i_H (H^{n-1}_\dR(X))\,.
\]
\end{corollary}

\begin{proof}
This follows from Theorem~\ref{thm:CalabiYau} and Ogus's result~\cite[Theorem~1]{Ogus}.
\end{proof}

Note that the integer $a$ appearing above is the $a$ number defined by van~der~Geer and Katsura~\cite{GK} for the special case of a Calabi-Yau family.

\section{Quartic hypersurfaces in $\PP^2$}

Our techniques also yield substantive information for hypersurfaces other than Calabi-Yau hypersurfaces; as an example, we include here the case of quartic hypersurfaces in~$\PP^2$.

When $f$ defines a Calabi-Yau hypersurface $X$, it is readily seen that the Frobenius action on the vector space $H^{\dim X}(X,\calO_X)$ is injective if and only if $\fpt(f)=1$, i.e., if and only
\[
\fpt(f)=\lct(f)\,.
\]
For hypersurfaces $X$ of general type, the injectivity of the Frobenius on $H^{\dim X}(X,\calO_X)$, or even the ordinarity of $X$ in the sense of Bloch and Kato~\cite[Definition~7.2]{Bloch-Kato}---a stronger condition---does not imply the equality of the $F$-pure threshold and the log canonical threshold: for example, for each $f$ defining a quartic hypersurface in $\PP^2$ over a field of characteristic $p\equiv 3\mod 4$, we shall see that $\fpt(f)<\lct(f)$; we emphasize that this includes the case of generic hypersurfaces, and that these are ordinary in the sense of Bloch and Kato by a result of Deligne; see~\cite{Illusie}. More generally:

\begin{lemma}
\label{lemma:pigeon}
Let $f$ be a homogeneous polynomial of degree $d$ in $K[x_0,\dots,x_n]$. Then:
\begin{enumerate}[\quad\rm(1)]
\item For each $q=p^e$, one has $\displaystyle{\mu_f(q)\ \le\ \left\lceil\frac{nq+q-n}{d}\right\rceil}\,,$

\item If $nq+q$ is congruent to any of $1,2,\dots,n\mod d$ \ for some $q$, then $\displaystyle{\fpt(f)<\frac{n+1}{d}}$.
\end{enumerate}
\end{lemma}

For quartics in $\PP^2$, one has $n=2$ and $d=4$; thus, if $p\equiv 3\mod 4$, then $np+p\equiv1\mod d$, so the $F$-pure threshold is strictly smaller than the log canonical threshold by (2).

\begin{proof}
The pigeonhole principle implies that $f^k\in\frakm^{[q]}$ whenever $dk\ge (n+1)(q-1)+1$, which proves (1). For (2), suppose $nq+q$ is congruent to any of $1,2,\dots,n\mod d$. Then
\[
\left\lceil\frac{nq+q-n}{d}\right\rceil\ \le \ \frac{nq+q-n+(n-1)}{d}\ =\ \frac{nq+q-1}{d}\,,
\]
and it follows using (1) that $\mu_f(q)<(nq+q)/d$. Thus, 
\[
\fpt(f)\ \le\ \mu_f(q)/q\ <\ (n+1)/d\,.\qedhere
\]
\end{proof}

\begin{theorem}
\label{theorem:quartic}
Let $K$ be a field of characteristic $p> 2$. Let $f$ be a homogeneous polynomial of degree $4$ in $K[x_0,x_1,x_2]$, such that the Jacobian ideal of $f$ is $\frakm$-primary. Then the possible values for $\mu_f(q)$ and the $F$-pure threshold are:
\begin{align*}
p&\equiv 1\mod 4:\qquad \mu_f(q)=
\begin{cases}
\frac{q(3p-3)}{4p}&\text{ for all }q,\quad\fpt(f)=\frac{3p-3}{4p}\,,\\
\frac{3q+1}{4}&\text{ for all }q,\quad\fpt(f)=\frac{3}{4}\,,\\
\end{cases}\\
p&\equiv 3\mod 4:\qquad \mu_f(q)=
\begin{cases}
\frac{q(3p-5)}{4p}&\text{ for all }q,\quad\fpt(f)=\frac{3p-5}{4p}\,,\\
\frac{q(3p-1)}{4p}&\text{ for all }q,\quad\fpt(f)=\frac{3p-1}{4p}\,.\\
\end{cases}
\end{align*}
\end{theorem}

\begin{proof}
Lemma~\ref{lemma:bound} and Lemma~\ref{lemma:pigeon}\,(1) provide the respective inequalities
\[
\left\lceil\frac{3p-5}{4}\right\rceil\ \le\ \mu_f(p)\ \le\ \left\lceil\frac{3p-2}{4}\right\rceil\,.
\]

If $p\equiv 3\mod 4$, this reads
\[
\frac{3p-5}{4}\ \le\ \mu_f(p)\ \le\ \frac{3p-1}{4}\,,
\]
so there are two possible values for the integer $\mu_f(p)$. The sequence $\{\mu_f(q)/q\}_q$ is constant by Lemma~\ref{lemma:patterns}\,(2), which completes the proof in this case.

If $p\equiv 1\mod 4$, the inequalities read
\[
\frac{3p-3}{4}\ \le\ \mu_f(p)\ \le\ \frac{3p+1}{4}\,.
\]
Again, there are two choices for $\mu_f(p)$. If $\mu_f(p)=(3p-3)/4$, then $\{\mu_f(q)/q\}_q$ is a constant sequence by Lemma~\ref{lemma:patterns}\,(2), whereas if $\mu_f(p)=(3p+1)/4$, then Lemma~\ref{lemma:patterns}\,(1) implies that $\mu_f(q)=(3q+1)/4$.
\end{proof}

\begin{remark}
Similarly, for a quintic $f$ in $K[x_0,x_1,x_2]$ with an $\frakm$-primary Jacobian ideal, the possibilities for $\fpt(f)$ are easily determined; we do not list the corresponding $\mu_f(q)$ as these are dictated by $\fpt(f)$. We assume below that $p>5$.
\begin{alignat*}7
p&\equiv1\mod 5:\quad && \fpt(f):\quad &&(3p-3)/5p\,,\ && \text{ or }\ &&4/5\,,\\
p&\equiv2\mod 5:\quad && \fpt(f):\quad &&(3p-6)/5p\,,\ && \text{ or }\ &&(3p-1)/5p\,,\\
p&\equiv3\mod 5:\quad && \fpt(f):\quad &&(3p-4)/5p\,,\ && \text{ or }\ &&(3p^2-7)/5p^2\,,\ && \text{ or }\ &&(3p^2-2)/5p^2\,,\\
p&\equiv4\mod 5:\quad && \fpt(f):\quad &&(3p-7)/5p\,,\ && \text{ or }\ &&(3p-2)/5p\,.
\end{alignat*}
\end{remark}

The log canonical threshold of a smooth quartic in $\PP^2$ is $3/4$; except for the case where it equals $3/4$, the denominator of $\fpt(f)$ in Theorem~\ref{theorem:quartic} is $p$. For a quintic as above, if $\fpt(f)\neq\lct(f)$, then the denominator of $\fpt(f)$ is a power of $p$. More generally:

\begin{proposition}
Let $K$ be a field of characteristic $p>0$. Let $f$ be a homogeneous polynomial of degree $d$ in $K[x_0,\dots,x_n]$ with an $\frakm$-primary Jacobian ideal. If $p\ge nd-d-n$, then either $\fpt(f)=(n+1)/d$, or else the denominator of $\fpt(f)$ is a power of $p$.
\end{proposition}

\begin{proof}
If $\fpt(f)<(n+1)/d$, then there exists an integer $q$ such that $\mu_f(q)/q<(n+1)/d$. But then $\fpt(f)=\mu_f(q)/q$ by Lemma~\ref{lemma:patterns}\,(2).
\end{proof}

\section*{Acknowledgments}

We thank Mircea Musta\c t\u a for raising the question for elliptic curves at the AMS-MRC workshop on Commutative Algebra, Snowbird, 2010, the workshop participants, and the American Mathematical Society. We also thank Johan de Jong for a useful conversation.


\end{document}